\documentclass{article}
\usepackage[utf8]{inputenc}
\usepackage{hyperref}
\hypersetup{
    colorlinks=true,
    linkcolor=blue,
    citecolor=blue,      
    urlcolor=blue,
}
\usepackage{amsfonts}
\usepackage{amsrefs}
\usepackage{mathrsfs,amsmath}
\usepackage{amssymb}
\usepackage{titlesec}
\usepackage{graphicx}
\usepackage{wrapfig}
\usepackage{geometry}
\usepackage{amsthm}
\usepackage{commath}
\usepackage{color}
\newtheorem{theorem}{Theorem}[section]
\newtheorem*{remark}{Remark}
\newtheorem{lemma}[theorem]{Lemma}
\newtheorem*{conjecture}{Conjecture}

\title{Flip Distance and Triangulations of a Ball}
\author{Zili Wang}

\begin{document}
\maketitle

\begin{abstract}
   It is known that the flip distance between two triangulations of a convex polygon is related to the minimum number of tetrahedra in the triangulation of some polyhedron. It is interesting to know whether these two numbers are the same. In this work, we find examples to show that the two numbers are different in nature, and their ratio can be arbitrarily close to $\frac{3}{2}$.
\end{abstract}

\section{Introduction}

Let $\Sigma$ be a convex polygon, and $\mathcal{T}$ be a triangulation of $\Sigma$. Fix a diagonal $d$ of $\mathcal{T}$. Then $d$ is a diagonal of some convex quadrilateral in $\mathcal{T}$. A \textit{flip} of $d$ is the replacement of $d$ by the other diagonal in this quadrilateral. Two triangulations of $\Sigma$ differ by a flip if one of them is obtained by flipping a diagonal in the other. Flips between triangulations in a convex polygon have another interpretation as rotations of binary trees \cite{STT}. Although not ellaborated in this article, flips are generalized and defined on triangulations of topological surfaces \cite{GL}\cite{PP} and higher dimensional polytopes \cite{LRS}\cite{Pa}.

Given two triangulations $\mathcal{T}^+,\mathcal{T}^-$ of the convex polygon $\Sigma$, a \textit{flip path} from $\mathcal{T}^+$ to $\mathcal{T}^-$ is a sequence of triangulations starting at $\mathcal{T}^+$ and stopping at $\mathcal{T}^-$ such that every two consecutive triangulations differ by a flip. The length of a shortest flip path from $\mathcal{T}^+$ to $\mathcal{T}^-$ is their \textit{flip distance}, denoted by $\text{flip}(\mathcal{T}^+,\mathcal{T}^-)$. 

It is well-known that flip paths are related to triangulations of a ball. Note that $\mathcal{T}^+$ and $\mathcal{T}^-$ determine a topological triangulation $\tau$ of the sphere, obtained by placing $\mathcal{T}^+$ on top of $\mathcal{T}^-$ and gluing them along the boundary of $\Sigma$. A flip path from $\mathcal{T}^+$ to $\mathcal{T}^-$ determines a topological triangulation of the ball $\mathscr{T}$ whose boundary is $\tau$. We call $\mathscr{T}$ a \textit{tetrahedral decomposition extending} $\tau$ (see Section 2 and \cite{STT} for a more detailed description). 

It is then natural to quest for the relationship between the shortest flip paths from $\mathcal{T}^+$ to $\mathcal{T}^-$ and the \textit{minimal} tetrahedral decompositions extending $\tau$, where the number of tetrahedra is the smallest possible.  As the former gives rise to a tetrahedral decompositions extending $\tau$,  $\text{flip}(\mathcal{T}^+,\mathcal{T}^-)$ is no less than $\text{tet}(\tau)$, the number of tetrahedra in a minimal decomposition extending $\tau$. However, not much else is known about the relationship between the two numbers.

In \cite{STT} and an unpublished work by Claire Mathieu and William Thurston, the authors made an attempt to estimate $\text{flip}(\mathcal{T}^+,\mathcal{T}^-)$ using the Gromov norm of simplicial $3$-chains. The idea is to find the minimum $L^1$-norm of a $3$-chain whose boundary is $\tau$ (as a $2$-cycle), which converts the question into a problem in linear programming. The number $\text{tet}(\tau)$ is between this minimum norm and $\text{flip}(\mathcal{T}^+,\mathcal{T}^-)$, hence should be closer to $\text{flip}(\mathcal{T}^+,\mathcal{T}^-)$ than this minimum norm is.

In this work, we prove a result in the relationship between the two numbers:

\begin{theorem}\label{thm}
There exist triangulations $\mathcal{T}^+,\mathcal{T}^-$ of a convex polygon such that $\frac{\text{flip}(\mathcal{T}^+,\mathcal{T}^-)}{\text{tet}(\tau)}$ is arbitrarily close to $\frac{3}{2}$.
\end{theorem}

Thus, there can be a relatively large gap between $\text{tet}(\tau)$ and $\text{flip}(\mathcal{T}^+,\mathcal{T}^-)$. In the end of this article, we propose a conjecture on when these two numbers might be equal.



We will introduce background and notations in Section 2, and prove Theorem \ref{thm} in Section 3.

\section*{Acknowledgements}
The author is very thankful to Peter Doyle and Lionel Pournin for their very helpful discussions and comments.

\section{Background and Notations}

Let $\mathcal{T}^+, \mathcal{T}^-$ be triangulations of a convex polygon $\Sigma$. We are interested in studying the flip distance between $\mathcal{T}^+$ and $\mathcal{T}^-$. By the following Lemma proved in \cite{STT}, we can focus on the case when $\mathcal{T}^+$ and $\mathcal{T}^-$ do not contain a common diagonal:
 
 \begin{lemma}\label{lem2.1}
 If $\mathcal{T}^+$ and $\mathcal{T}^-$ have a diagonal in common, this diagonal is never flipped in any shortest flip path from $\mathcal{T}^+$ to $\mathcal{T}^-$.
 \end{lemma}

Denote by $XY$ the diagonal of $\Sigma$ joining vertices $X,Y$, and by $XYZW$ the quadrilateral with vertices $X,Y,Z,W$. If $XYZW$ is a convex quadrilateral with diagonals $XY$ and $ZW$, we can associate the flip from $XY$ to $ZW$ to a \textit{flipping tetrahedron}, obtained by gluing two copies of the quadrilateral $XYZW$ along the boundary, then drawing $XY$ on the front and $ZW$ on the back quadrilateral. 

Let $\tau$ be the triangulation of the sphere obtained by gluing $\mathcal{T}^+$ and $\mathcal{T}^-$ along the boundary of $\Sigma$.  Given a flip path from $\mathcal{T}^+$ to $\mathcal{T}^-$, we construct the flipping tetrahedron for every flip in the path, place the tetrahedra according to the locations of their vertices in $\Sigma$ and stack them from the top to the bottom in the order given by the flip path. The tetrahedra are then glued according to the way they are stacked. As $\mathcal{T}^+$ and $\mathcal{T}^-$ do not have a diagonal in common, the resulting space after gluing is homeomorphic to a ball, and the boundary consists exactly of triangles in $\tau$. This gives a tetrahedral decomposition extending $\tau$, thus $\text{tet}(\tau)\leq\text{flip}(\mathcal{T}^+,\mathcal{T}^-)$. 
 
 We say that a diagonal of $\Sigma$ is \textit{in a flip path} from $\mathcal{T}^+$ to $\mathcal{T}^-$ if it belongs to a triangulation in this flip path. 
 If a diagonal is in the flip path but not in $\mathcal{T}^+$ or $\mathcal{T}^-$, we call this diagonal \textit{extra}. The following lemma relates the length of a flip path to the number of extra diagonals in the path.
 
\begin{lemma}\label{lem2.2}
Let $e$ be the number of extra diagonals in a flip path, and $n$ be the number of vertices of the polygon. Then the number of flips is $n-3+e$.
\end{lemma}
\begin{proof}
$\text{number of flips}=\text{number of diagonals in }\mathcal{T}^++\text{number of extra diagonals}=n-3+e$.
\end{proof}

Finally, to prove Theorem \ref{thm}, we need the following lemma proved in \cite{PW}:
 


\begin{lemma}\label{lem2.3}
Suppose that three diagonals $d_1,d_2,d_3$ are in a shortest flip path from $\mathcal{T}^+$ to $\mathcal{T}^-$ and they form a triangle. Then this triangle belongs to a triangulation in this flip path.
\end{lemma}

\section{Proof of Theorem 1.1}

Let $\Sigma_n$ be the convex polygon with $2n+4$ vertices ($n\geq 2$) labelled in Figure \ref{fig1} (left). Let $\mathcal{T}^+_n$ and $\mathcal{T}^-_n$ be the triangulations of $\Sigma_n$ drawn in Figure \ref{fig1} (right). Note that some vertices are omitted in the figure.

\begin{figure}[h]
\centering
\includegraphics[width=0.71\textwidth]{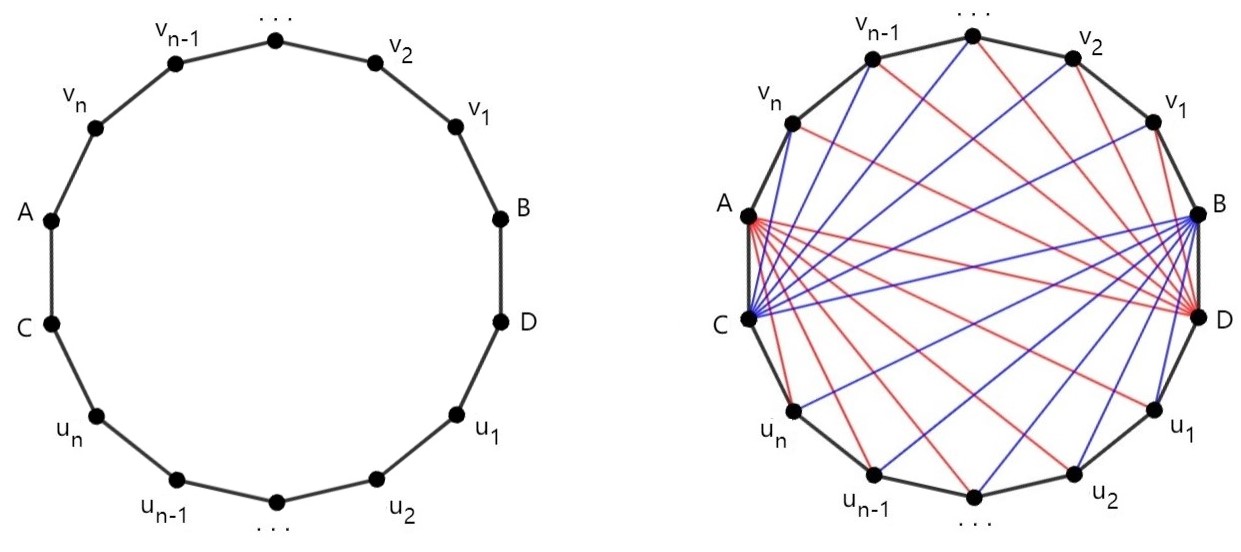}
\caption{$\Sigma_n$ (left) and its triangulations (right): $\mathcal{T}^+_n$ (blue) and $\mathcal{T}^-_n$ (red)}
\label{fig1}
\end{figure}

We first find an upper bound for $\text{flip}(\mathcal{T}^+_n,\mathcal{T}^-_n)$.

\begin{lemma}\label{lem3.1}
$\text{flip}(\mathcal{T}^+_n,\mathcal{T}^-_n)\leq 3n+1$
\end{lemma}
\begin{proof}
We can transform $\mathcal{T}^+_n$ to the triangulation in the left of Figure \ref{fig2} by performing a sequence of $n$ flips in the order of diagonals $Cv_n, Cv_{n-1},\dots,Cv_1$. From this triangulation, we perform another $n+1$ flips, in the order of diagonals $BC$ followed by $Bu_n,Bu_{n-1},\dots,Bu_1$, and obtain the triangulation in the right of Figure \ref{fig2}. Finally, an additional $n$ flips take this triangulation to $\mathcal{T}^-_n$.
\end{proof}

\begin{figure}[h]
\centering
\includegraphics[width=0.71\textwidth]{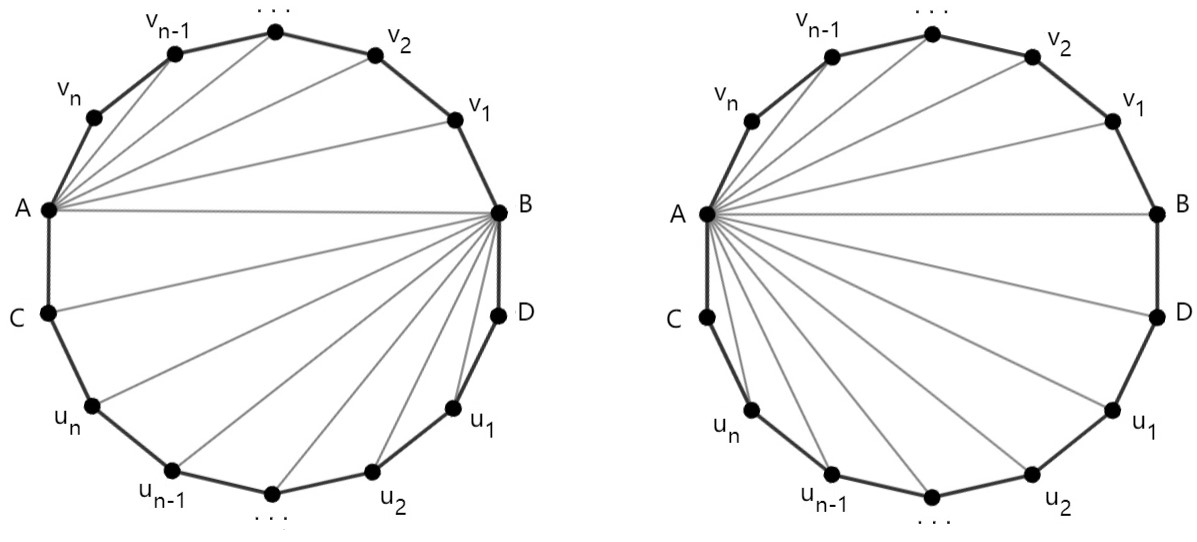}
\caption{Triangulations obtained from  $\mathcal{T}^+_n$ after $n$ flips (left) and then $n+1$ more flips (right).}
\label{fig2}
\end{figure}

From now on, we fix a shortest flip path from  $\mathcal{T}^+_n$ to $\mathcal{T}^-_n$ and denote this path by $\mathcal{P}$. The goal is to show that $\mathcal{P}$ contains exactly $3n+1$ flips. We first count the number of diagonals in $\mathcal{T}^+_n$ that are directly flipped to a diagonal in $\mathcal{T}^-_n$ in $\mathcal{P}$.

\begin{lemma}\label{lem3.2}
In the shortest flip path $\mathcal{P}$, at least $n+1$ diagonals in $\mathcal{T}^+_n$ are flipped directly to some diagonal in $\mathcal{T}^-_n$.
\end{lemma}

\begin{proof}
Let $x$ be the number of diagonals in $\mathcal{T}^+_n$ that are flipped directly to a diagonal in $\mathcal{T}^-_n$ in $\mathcal{P}$. 
Then $x$ flips are required in total to transform all these diagonals to diagonals in $\mathcal{T}^-_n$. For each of the remaining diagonals, at least $2$ flips are required to transform it to a diagonal in $\mathcal{T}^-_n$. There are $2n+1-x$ such remaining diagonals. Thus, the total number of flips in this path is no less than $x+2(2n+1-x)=4n+2-x$. From Lemma \ref{lem3.1}, $4n+2-x\leq 3n+1$, so $x\geq n+1$.
\end{proof}

\begin{lemma}\label{lem3.3}
$AB$ or $CD$ is in $\mathcal{P}$. 
\end{lemma}
\begin{proof}
 In the proof, a flip in $\mathcal{P}$ from $d$ to $d'$ will be denoted by $d\to d'$ if $d$ is a diagonal in $\mathcal{T}^+_n$ and $d'$ is in $\mathcal{T}^-_n$. Note that $d$ must cross $d'$. 

Depending on the vertices of $d$ and $d'$, a flip $d\to d'$ belongs to one of the following four types:



\textbf{Type 1.} $BX\to DY$, where $X=C$ or $u_i$, $Y=A$ or $v_j\,(1\leq i,j\leq n)$

\textbf{Type 2.} $BX\to AY$, where $X=C$ or $u_i$, $Y=D$ or $u_j$, where $i>j$ if $X=u_i$ and $Y=u_j$

\textbf{Type 3.} $CX\to AY$, where $X=B$ or $v_i$, $Y=D$ or $u_j\,(1\leq i,j\leq n)$

\textbf{Type 4.} $CX\to DY$, where $X=B$ or $v_i$, $Y=A$ or $v_j$, where $i<j$ if $X=u_i$ and $Y=u_j$

If there is a flip of Type 2 in $\mathcal{P}$, then there exists a triangulation in $\mathcal{P}$ containing the quadrilateral $ABXY$, and hence the diagonal $AB$ is in $\mathcal{P}$. Similarly, if there is a flip of Type 4, then $CD$ is in $\mathcal{P}$. Thus, it suffices to prove the lemma when all flips of the form $d\to d'$ have Type 1 or 3.  

If there is a flip of Type 1, then $BY$ and $DX$ are in $\mathcal{P}$. If $Y=A$ or $X=C$, we are done. So we assume that each flip of Type 1 has the form $Bu_i\to Dv_j$, and associate this flip to diagonals $Bv_j$ and $Du_i$ which are in $\mathcal{P}$. Similarly, we assume each flip of Type 3 has the form $Cv_j\to Au_i$ and associate this flip to diagonals $Cu_i$ and $Av_j$ which are in $\mathcal{P}$. 

Note that the diagonals associated to Type 1 and Type 3 flips are mutually distinct. These are all extra diagonals except possibly $Bv_1, Du_1, Cu_n$ and $Av_n$, which are the edges of $\Sigma_n$. Thus, by Lemma \ref{lem3.2} and our assumption, there are at least $2(n+1)-4=2n-2$ extra diagonals in $\mathcal{P}$. But by our assumption again, the flips $Bu_i\to Dv_j$ or $Cv_j\to Au_i$ will introduce at least an extra diagonal $u_iv_j$ for some $i$ and $j$. Consequently, there are at least $2n-2+1=2n-1$ extra diagonals in $\mathcal{P}$.

By Lemma \ref{lem2.2}, $\mathcal{P}$ has at least $(2n+4)-3+(2n-1)=4n$ flips. But $4n>3n+1$ since $n\geq 2$. This is a contradiction to the minimality in length of $\mathcal{P}$ by Lemma \ref{lem3.1}. Thus, the assumption is wrong, and $AB$ or $CD$ is in $\mathcal{P}$.
\end{proof}

\begin{lemma}\label{lem3.4}
$\text{flip}(\mathcal{T}^+_n,\mathcal{T}^-_n)=3n+1$.
\end{lemma}
\begin{proof}
By Lemma \ref{lem3.3}, $AB$ or $CD$ is in $\mathcal{P}$. We assume that $AB$ is in $\mathcal{P}$, and the other case can be proved by symmetry.


Note that $AB,AD$ and $BD$ form a triangle. By Lemma \ref{lem2.3}, $\triangle ABD$ belongs to a triangulation $\mathcal{T}$ in $\mathcal{P}$. 

Let $\Sigma_v$ and $\Sigma_u$ be the polygons obtained by cutting $\Sigma$ along $AD$, where $\Sigma_v$ contains $v_j$'s and $\Sigma_u$ contains $u_i$'s. Consider the sub-path of $\mathcal{P}$ from $\mathcal{T}$ to $\mathcal{T}^-_n$. Since $\mathcal{T}$ and $\mathcal{T}^-_n$ both contain $AD$, by Lemma \ref{lem2.1}, the diagonals flipped in this sub-path are in $\Sigma_v$ or $\Sigma_u$. Thus, we can rearrange the order of triangulations in this sub-path so that all diagonals in $\Sigma_u$ are flipped before any diagonal in $\Sigma_v$. When all diagonals in $\Sigma_u$ are flipped, $\mathcal{T}$ is transformed into a triangulation $\mathcal{T}'$ containing $\triangle ABD$ and all the diagonals $Au_i\,(1\leq i\leq n)$, as in Figure \ref{fig3}. Since this modified flip path is also shortest, $\text{flip}(\mathcal{T}^+_n,\mathcal{T}^-_n)=\text{flip}(\mathcal{T}^+_n,\mathcal{T}')+\text{flip}(\mathcal{T}',\mathcal{T}^-_n)$.  

\begin{figure}[h]
\centering
\includegraphics[width=1.0\textwidth]{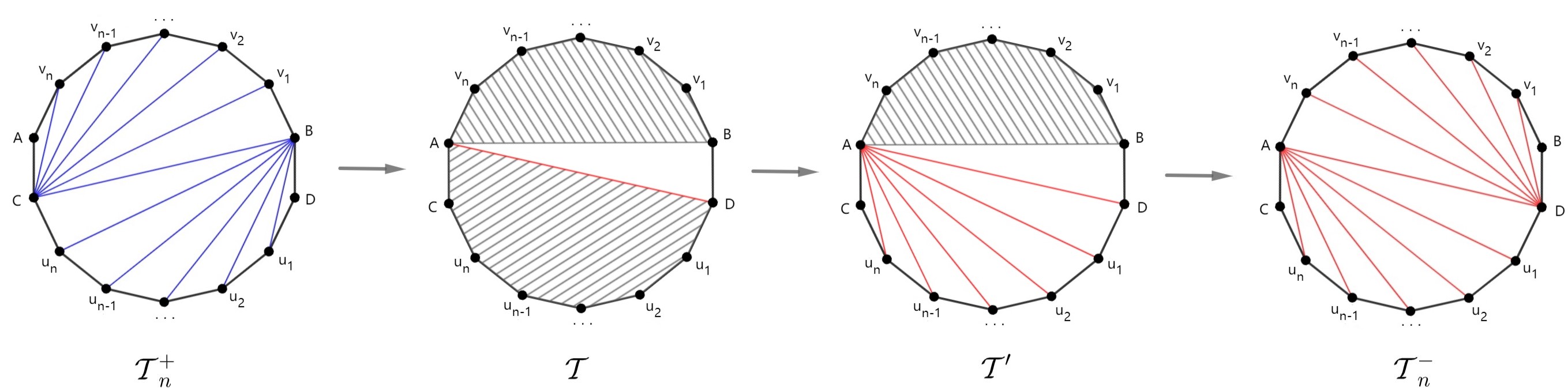}
\caption{A modified shortest flip path from $\mathcal{T}^+_n$ to $\mathcal{T}^-_n$ passing through $\mathcal{T}'$. Each shaded region represents some triangulation, and the two shaded regions with the same pattern represent the same triangulation.}
\label{fig3}
\end{figure}

To compute $\text{flip}(\mathcal{T}^+_n,\mathcal{T}')$,  consider an arbitrary flip path from $\mathcal{T}^+_n$ to $\mathcal{T}'$. Note that all diagonals in $\mathcal{T}^+_n$ cross some diagonals in $\mathcal{T}'$. Since there are $2n+1$ diagonals in $\mathcal{T}^+_n$, at least $2n+1$ flips are required in the flip path to move all of them. Hence $\text{flip}(\mathcal{T}^+_n,\mathcal{T}')\geq 2n+1$.

Similarly, to compute $\text{flip}(\mathcal{T}',\mathcal{T}^-_n)$, consider an arbitrary flip path from $\mathcal{T}^-_n$ to $\mathcal{T}'$. Since all the diagonals $Dv_j\,(1\leq j\leq n)$ in $\mathcal{T}^-_n$ cross $AB$ in $\mathcal{T}'$, at least $n$ flips in this path are required to move all the $Dv_j$'s. Therefore, $\text{flip}(\mathcal{T}',\mathcal{T}^-_n)\geq n$. 

Combined, we have $\text{flip}(\mathcal{T}^+_n,\mathcal{T}^-_n)=\text{flip}(\mathcal{T}^+_n,\mathcal{T}')+\text{flip}(\mathcal{T}',\mathcal{T}^-_n)\geq (2n+1)+n=3n+1$. Together with Lemma \ref{lem3.1}, we get $\text{flip}(\mathcal{T}^+_n,\mathcal{T}^-_n)=3n+1$.
\end{proof}

Let $\tau_n$ be the topological triangulation of the sphere determined by $\mathcal{T}^+_n$ and $\mathcal{T}^-_n$. We will show that $\text{tet}(\tau_n)= 2n+3$. 

Take a different Hamiltonian cycle in $\tau_n$ so that we can cut $\tau_n$ along the cycle into two triangulations that are topologically $\widetilde{\mathcal{T}}^+_n$ and $\widetilde{\mathcal{T}}^-_n$ in Figure \ref{fig4} (left). Observe from Figure \ref{fig5} (left) that there is a flip path from $\widetilde{\mathcal{T}}^+_n$ to $\widetilde{\mathcal{T}}^-_n$ containing only two extra diagonals $AB$ and $CD$, hence $2n+3$ flips by Lemma \ref{lem2.2}. Denote this flip path by  $\widetilde{\mathcal{P}}$. Then  $\widetilde{\mathcal{P}}$ gives rise to a tetrahedral decomposition extending $\tau_n$, so $\text{tet}(\tau_n)\leq 2n+3$.

\begin{figure}[h]
\centering
\includegraphics[width=0.72\textwidth]{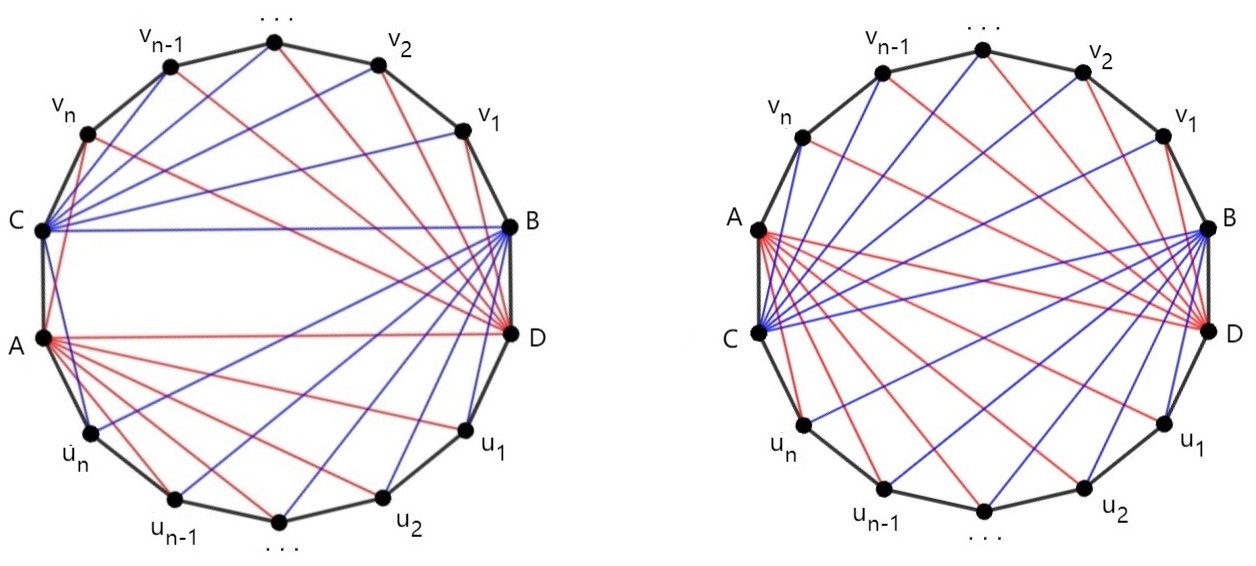}
\caption{Left: $\widetilde{\mathcal{T}}^+_n$ (blue) and $\widetilde{\mathcal{T}}^-_n$ (red). Right: $\mathcal{T}^+_n$ (blue) and $\mathcal{T}^-_n$ (red). Both pairs of triangulations compose $\tau_n$. Note that the vertex-labels are different in the two pictures.}
\label{fig4}
\end{figure}

\begin{figure}[h]
\centering
\includegraphics[width=0.72\textwidth]{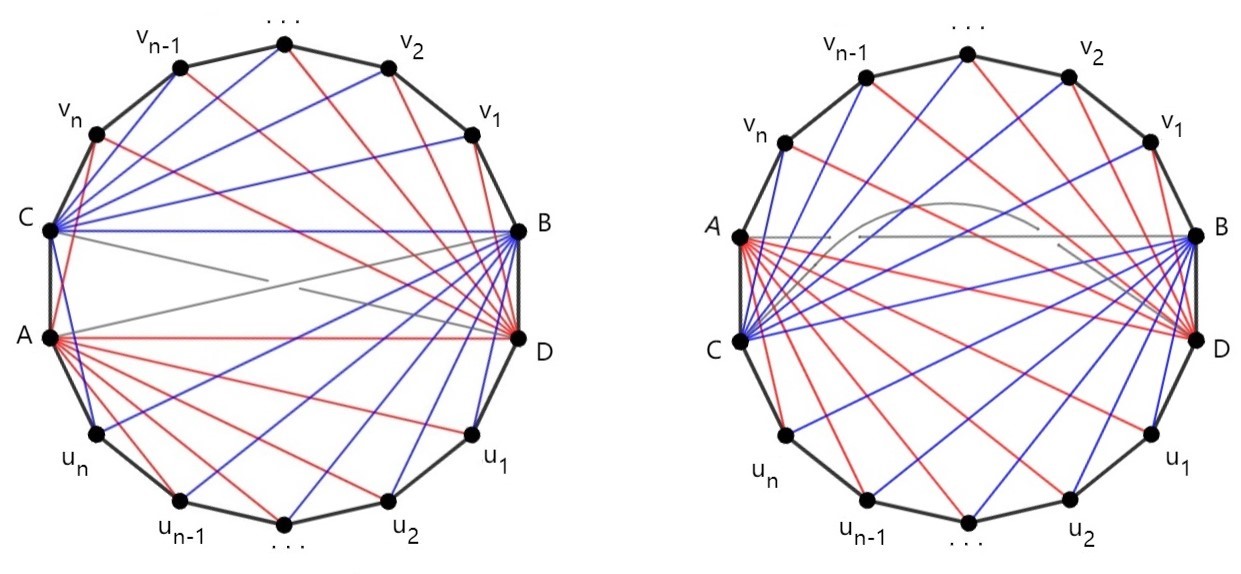}
\caption{Left: a flip path from $\widetilde{\mathcal{T}}^+_n$ to $\widetilde{\mathcal{T}}^-_n$ containing two extra diagonals $AB$ and $CD$. Right: the triangulation with vertices $A,B,C,D$ in the minimal decomposition is bent to preserve topology.}
\label{fig5}
\end{figure}

\begin{lemma}\label{lem3.5}
$\text{tet}(\tau_n)= 2n+3$.
\end{lemma}
\begin{proof}
We first show that $\text{tet}(\tau_n)> 2n+1$. Note that $\text{tet}(\tau_n)$ has $4n+4$ faces. If there exists a tetrahedral decomposition extending $\tau_n$ with $\leq 2n+1$ tetrahedra, at least one tetrahedron contains three faces in $\tau_n$, but there is no such tetrahedron by observation.

If there is a tetrahedral decomposition extending $\tau_n$ with $2n+2$ tetrahedra, since no three faces in $\tau_n$ belong to one tetrahedra, each tetrahedron must contain exactly two faces in $\tau_n$, and the remaining two faces are shared with two other tetrahedra. Thus, there are $2(2n+2)+(2n+2)=6n+6$ faces in this tetrahedral decomposition. By computing the Euler characteristics, the number of edges in this tetrahedral decomposition is $6n+7$. Since $\tau_n$ has $6n+6$ edges, there is only one edge not in $\tau_n$. This edge is shared by all tetrahedra in the decomposition, implying that $\tau_n$ is a bipyramid, a contradiction.
\end{proof}

Combining Lemma \ref{lem3.4} and Lemma \ref{lem3.5}, we see that $\text{flip}(\mathcal{T}^+_n,\mathcal{T}^-_n)> \text{tet}(\tau_n)$ when $n>2$, and Theorem \ref{thm} is proved by taking the limit as $n$ goes to infinity. Figure \ref{fig5} (right) shows that this minimal tetrahedral decomposition extending $\tau_n$ does not give rise to a flip path from $\mathcal{T}^+_n$ to $\mathcal{T}^-_n$, since one tetrahedron must be bent to preserve the topology and cannot be a flipping tetrahedron.

\begin{remark}
So far, these are the most extreme examples we constructed in an attempt to maximize $\frac{\text{flip}(\mathcal{T}^+,\mathcal{T}^-)}{\text{tet}(\tau)}$. We conjecture that it is impossible for this ratio to exceed or even reach $\frac{3}{2}$. In addition, if there is only a small difference in the maximum and minimum degrees of vertices in $\tau$, then ${\text{flip}(\mathcal{T}^+,\mathcal{T}^-)}$ and ${\text{tet}(\tau)}$ should be pretty close.
\end{remark}

\section*{Some Additional Thoughts}

It is known that the diameter of the flip graph of a convex $n$-gon is $2n-10$ for $n>12$. The proof is combinatorical for all $n>12$ \cite{P} and involves hyperbolic geometry for large $n$ \cite{STT}. It would be interesting to look for pairs of triangulations $(\mathcal{T}^+_n,\mathcal{T}^-_n)$ achieving this diameter. Since \cite{P} focused on the existence of such examples, and \cite{STT} found some examples that worked only for large $n$ and involved computation of hyperbolic volume, it would be helpful to have some criteria that can generate larger number of such examples and are easier to check.

A \textit{cone-type triangulation} at a vertex $v$ is a triangulation in which every tetrahedron contains $v$ and a face in $\mathcal{\tau}$.  Cone-type triangulations can be converted to flip paths. Since $2n-10$ is the number of tetrahedra in a cone-type triangulation at a vertex of degree $6$, we propose that it is a minimal tetrahedral decomposition extending $\tau$ for certain $\tau$. More specifically, we make the following conjecture:

\begin{conjecture}

 Let $\tau$ be a spherical triangulation with vertices of degree $5$ and $6$ only. We say that a cycle in $\tau$ is bad if it has length $l$ and separates two vertices of degree $>l$ in $\tau$. If there is no bad cycle in $\tau$, then the cone-type triangulation at a degree-$6$ vertex is minimal, and $\text{flip}(\mathcal{T}^+_n,\mathcal{T}^-_n)=2n-10$.
\end{conjecture}

The absence of bad cycles is necessary to make the conjecture hold. Otherwise, no cone-type triangulation is minimal. For example, let $\tau$ be the triangulation of the sphere obtained by gluing two copies of the graph in Figure \ref{fig6} along the boundary. Then $\tau$ has $27$ vertices and a bad cycle of length $5$ (thickened in Figure \ref{fig6}).

\begin{figure}[h]
\centering
\includegraphics[width=0.39\textwidth]{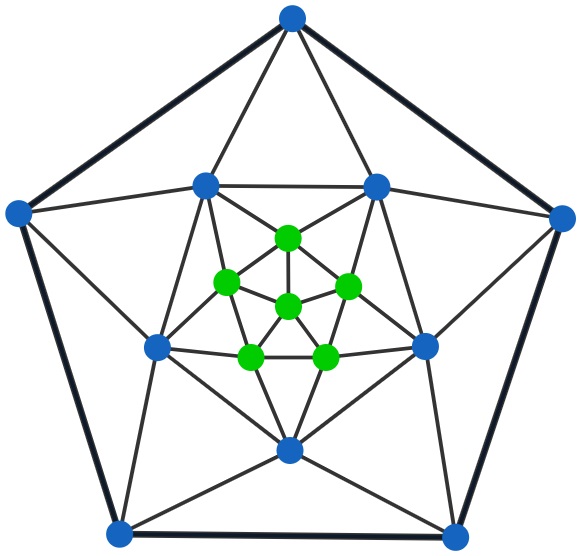}
\caption{Gluing two copies of this graph along the boundary gives rise to $\tau$. Vertices of degree $5$ and $6$ in $\tau$ are colored green and blue.}
\label{fig6}
\end{figure}

The cone-type triangulation at a degree $6$ vertex has $2n-10=44$ tetrahedra. However, we can create a tetrahedral decomposition extending $\tau$ with fewer tetrahedra: For each copy of this graph, choose a vertex $v$ of degree $6$ in the interior. Consider the set of tetrahedra containing $v$ and a face in this graph (and similarly for the other copy). Removing these tetrahedra from both copies yields a pentagonal bi-pyramid, which can be triangulated using $5$ tetrahedra. This tetrahedral decomposition uses $19\times 2+5=43$ tetrahedra, so cone-type triangulations are not minimal.

We thank Peter Doyle again for his advice in refining this conjecture.


\begin{thebibliography}{}

\bibitem{LRS} De Loera, J., Rambau, J., \& Santos, F. (2010). Triangulations: structures for algorithms and applications (Vol. 25). Springer Science \& Business Media.

\bibitem{GL} Gültepe, F., \& Leininger, C. (2017). An arc graph distance formula for the flip graph. Proceedings of the American Mathematical Society, 145(7), 3179-3184.

\bibitem{Pa} Pachner, U. (1991). PL homeomorphic manifolds are equivalent by elementary shellings. European journal of Combinatorics, 12(2), 129-145.

\bibitem{PP} Parlier, H., \& Pournin, L. (2017). Flip-graph moduli spaces of filling surfaces. Journal of the European Mathematical Society, 19(9), 2697-2737.

\bibitem{P} Pournin, L. (2014). The diameter of associahedra. Advances in Mathematics, 259, 13-42.

\bibitem{PW} Pournin, L., \& Wang, Z. (2021).  \href{https://arxiv.org/abs/2106.08012}{\textit{Strong convexity in flip-graphs}}. arXiv preprint arXiv:2106.08012.

\bibitem{STT} Sleator, D. D., Tarjan, R. E., \& Thurston, W. P. (1988). Rotation distance, triangulations, and hyperbolic geometry. Journal of the American Mathematical Society, 1(3), 647-681.


\end{thebibliography}
\end{document}